\documentclass[10pt]{article}
 \usepackage{amsmath,amsfonts,amssymb,amsthm}
 \usepackage{hyperref}

 \theoremstyle{plain}
 \newtheorem{theorem}{Theorem}[section]
 \newtheorem{lemma}[theorem]{Lemma}
 \newtheorem{corollary}[theorem]{Corollary}
 \newtheorem{proposition}[theorem]{Proposition}
 \newtheorem{conjecture}[theorem]{Conjecture}

 \theoremstyle{definition}
 \newtheorem{definition}[theorem]{Definition}

 \theoremstyle{remark}
 \newtheorem{remark}[theorem]{Remark}

\newcommand{\spec}{\operatorname{Spec}}

\newcommand{\chara}{\operatorname{char}}

\newcommand{\hilb}{\operatorname{Hilb}}

\newcommand{\mni}{\medskip\noindent}

\usepackage{graphicx}
\usepackage{xcolor}
\usepackage{float}
\usepackage{amssymb}
\usepackage{amsmath}
\usepackage{mathrsfs}
\usepackage{bm}
\usepackage[all,cmtip]{xy}
\usepackage{enumitem}
\usepackage[utf8]{inputenc}
\usepackage[english]{babel}
\usepackage{multirow}
\usepackage{mathtools}
\usepackage{scalerel,stackengine}
\usepackage{stmaryrd}
\usepackage{datetime}
\usepackage[title]{appendix}
\usepackage{url}
\usepackage{tabulary}
\usepackage{booktabs}
\usepackage{pifont}

\title{On Examples of Supersingular and Rationally Chain Connected Threefolds}

\title{%
  On Examples of Supersingular and Rationally Chain Connected Threefolds}

\author{Santai Qu}

\date{\today}

\begin{document}

\maketitle

\begin{abstract}
In this work, we show that for a certain class of threefolds in positive characteristics, 
rational-chain-connectivity is equivalent to supersingularity.  The same result is 
known for K3 surfaces with elliptic fibrations.  And there are examples of threefolds 
that are both supersingular and rationally chain connected.
\end{abstract}

\section{Introduction}

The goal of this article is to generalize the following result to a certain class of threefolds in positive characteristics.

\begin{theorem}(\cite{open}, Problem 12, p.11 and p.12)\label{openK3}
Let $k$ be an algebraically closed field in positive characteristic.  
If a K3 surface $X$ over $k$ admits an elliptic fibration, then $X$ is supersingular
if and only if any two points on $X$ can be connected by a chain of finitely rational curves.
\end{theorem}

To generalize Theorem~\ref{openK3} to threefolds, 
it is natural to consider the class of threefolds with K3-fibrations.
There are examples of Calabi-Yau threefolds with smooth proper K3-fibrations over $\mathbb{P}_k^1$ constructed by 
Stefan Schr\"oer in \cite{stefan}, which are both supersingular and unirational (Theorem~\ref{unirationalp3} and Theorem~\ref{unirational2}).
In this article, we will show that supersingularity is equivalent to 
rational-chain-connectivity for a broader class of threefolds that are generalizations of Schr\"oer's
examples.  

\subsection{Preliminaries}

To state our results, we first fix some conventions.  Let $k$ be an algebraically closed field
of characteristic $p>0$.  By a Calabi-Yau $n$-fold $X$, 
we mean a smooth, projective $k$-scheme of dimension $n$ such that $\mathrm{H}^i(X, \mathcal{O}_X)=0$
for $0<i<n$ and that $\omega_X\simeq \mathcal{O}_X$.  When $n$ equals two, a Calabi-Yau 2-fold 
is a K3 surface.  And when $n$ is one, a Calabi-Yau 1-fold is just an elliptic curve.

\begin{definition}\label{supersingularitydefinition}
Let $X$ be a smooth projective variety over a perfect field $k$ in characteristic $p>0$.
Let $W=W(k)$ be its Witt vectors.  Denote by $K$ the fraction field of $W$.
We say that $X$ is \emph{supersingular} if the slope $[0,1)$ part of $\mathrm{H}^i_{\text{rig}}(X/K)$ is vanishing for all $i>0$.
And we say that $X$ is \emph{$\mathrm{H}^i$-supersingular} is the slope 
$[0,1)$ part of $\mathrm{H}^i_{\text{rig}}(X/K)$ is vanishing.
\end{definition}

Note that for a smooth and proper variety $X$ over a perfect field $k$ with $\chara k>0$, 
the rigid cohomology $\mathrm{H}^i_{\text{rig}}(X/K)$ coincides with
the crystalline cohomology $\mathrm{H}^i_{\text{cris}}(X/W)$ on $X$. 
Let $n$ be the dimension of $X$.  Some authors use the
definition that $X$ is supersingular if it is $\mathrm{H}^n$-supersingular as in our definition (see \cite{YT}, Definition 13, p.8).
And $\mathrm{H}^n$-supersingularity is equivalent to saying that the formal Brauer group of the surface has infinite height,
see Corollary II.4.3 of \cite{artinmazur}.

One of reasons for the study of supersingularity is its connection with unirationality for K3 surfaces 
conjectured by Artin, Shioda, etc.

\begin{conjecture}(Artin, Rudakov, Shafarevich, Shioda)\label{artinconjecture}
Let $k$ be an algebraically closed field in characteristic $p>0$.  A K3 surface over $k$
is supersingular if and only if it is unirational.
\end{conjecture}

This conjecture holds for Kummer surfaces when $p\ge 3$ and for any K3 surfaces when 
$p=2$.  Note that any Kummer surface has an elliptic fibration.  
Motivated by Schr\"oer's examples, we make the following definition for K3-fibrations.

\begin{definition}\label{definitionK3fibration}
Let $X$ be a smooth and projective threefold over a field $k$.  
We say that $\pi: X\to \mathbb{P}_k^1$ is a \emph{K3-fibration}
if $\pi$ is smooth, proper and for every $s\in \mathbb{P}_k^1$ the geometric
fiber $X_{\overline{\kappa(s)}}$ is a K3 surface.
\end{definition}

The part that any two points on the K3 surface $X$ can be joint by a chain of rational curves in Theorem~\ref{openK3} is
about the rational-chain-connectivity of $X$.  In general, let $X$ be a variety over a field $k$.  
Recall that a \emph{tree of rational curves $C$} on $X$ is a $k$-morphism 
$u_C: C\to X$ where $C$ is a proper, geometrically connected curve of arithmetic genus zero, defined over a field extension $k_C$ of $k$,
and every irreducible component of $C_{\overline{k_C}}$ is rational.   We say that two distinct points
$x_1$ and $x_2$ on $X$, not necessarily closed, are connected by the tree of rational curve $C$ 
if they are in the imgae of $C_{\overline{k_C}}\to X_{\overline{k}}\to X$.
Moreover, we say that $C$ is a \emph{rational curve} if $C_{\overline{k_C}}$ is irreducible and rational.

\begin{definition}\label{rationallychainconnecteddefinition}
Let $X$ be a variety over a field $k$.  
We say that $X$ is \emph{rationally chain connected (resp. rationally connected)} if every pair of distinct general 
points $x_1$, $x_2$ on $X$, including non-closed points, are connected by a tree
of rational curves (resp. by a rational curve).
\end{definition}

Note that Definition~\ref{rationallychainconnecteddefinition} agrees with Definition (3.2)
in \cite{Kollar}, p.199, since we include non-closed points in the definition. 
When $k$ is an uncountable algebraically closed field, rational-chain-connectivity (resp. rational-connectivity)
 is equivalent to that every pair of two distinct very general closed points could be connected by a connected tree of 
rational curves (resp. by a rational curve), see \cite{Kollar}, Proposition (3.6), p.201.  

Now, we can state our main theorem.

\begin{theorem}\label{maintheorem}
Fix an algebraically closed field $k$ with $\chara k\ge 3$.
Let $X$ be a smooth, projective threefold over $k$ with a K3-fibration.  
Denote by $k(t)$ the function field of $\mathbb{P}_k^1$.
Suppose that $X_{k(t)}$ is a Kummer surface.  
Then, the following are equivalent:
\begin{enumerate}
\item{the slopes of \(\mathrm{H}^{2}_{\text{rig}}(X/K)\) are all \(1\),}
\item{$X$ is supersingular,}
\item{$X$ is rationally chain connected.}
\end{enumerate}
In particular, if $X$ is supersingular, then $X$ is uniruled.
\end{theorem}

\subsection{Outline of the article}

In Section~\ref{rationallychainconnectedK3surfaces},
we show that rationally chain connected K3 surfaces are unirational if we assume Conjecture~\ref{artinconjecture} (Corollary~\ref{equivalencefork3surfaces}).
In Section~\ref{examplesunirational}, we go over the constructions of Calabi-Yau threefolds in 
Schr\"oer's article \cite{stefan} and show that they are all unirational (Theorem~\ref{unirationalp3} and Theorem~\ref{unirational2}).  
Section~\ref{rationallyconnectedfibrations} shows that it is sufficient to assume the rational-connectivity 
of the generic fiber of a fibration to obtain rational-chain-connectivity of other fibers (Proposition~\ref{rationallyconnectedgeometricfiber}).  
Section~\ref{technicalsection}
is the technical part about crystalline cohomology, which is used to show that every K3 fiber will be 
supersingular if we assume the total space is supersingular (Proposition~\ref{supersingularfibration}).  
Finally, in Section~\ref{proofofthemaintheorem},
we give the proof of Theorem~\ref{maintheorem}.

\mni
\noindent{\bf Acknowledgement:}  The author is very grateful to his advisor Prof. Jason Michael Starr for his consistent support during the proof.
The author really appreciates the help of Dingxin Zhang in the proof about crystalline cohomology.

\section{Rationally chain connected K3 surfaces}\label{rationallychainconnectedK3surfaces}

The definition of rational-chain-connectivity directly leads to following result.

\begin{theorem}(\cite{Kollar}, Theorem (3.13.1), p.206)\label{weakrationallychainconnectedK3}
Let $X$ be a rationally chain connected variety over an algebraically
closed field $k$.  Then, $A_0(X)\simeq \mathbb{Z}$.
\end{theorem}

There are several notions of supersingularity for K3 surfaces in literatures which are all equivalent.

\begin{definition}
Let $X$ be a K3 surface over an algebraically closed field.  We say that
$X$ is \emph{Shioda-supersingular} if its Picard rank is 22.
And $X$ is called \emph{Artin-supersingular} if the height of its Artin-Mazur 
formal Brauer group is infinity, that is, $h(\widehat{Br}(X))=\infty$.
\end{definition}

The following result characterizes Shioda-supersingularity and Artin-supersingularity
in terms of the slopes of crystalline cohomology.

\begin{theorem}(\cite{liedtke}, Theorem 2.3)\label{supersingularK3}
For a K3 surface $X$ over an algebraically closed field $k$ in odd characteristic.  
Denote by $W=W(k)$ the ring of Witt vectors of $k$.  Then, the following are equivalent:
\begin{enumerate}[label=(\roman*)]
\item{$X$ is Shioda-supersingular.}
\item{$X$ is Artin-supersingular.}
\item{For all $i$, the F-crystal $\mathrm{H}^i_{\text{cris}}(X/W)$ is of slope $i/2$.}
\item{The slope $[0,1)$ part of $\mathrm{H}^2_{\text{cris}}(X/W)$ is vanishing.}
\end{enumerate}
\end{theorem}

Therefore, we just say that a K3 surface $X$ is \emph{supersingular} if it satisfies one of
the conditions in Theorem~\ref{supersingularK3}.  

Fix a perfect field $k$ with characteristic $p>0$.   We cite the following result about vanishing of slope $[0,1)$
part of rigid cohomology.

\begin{theorem}(\cite{Esnault}, Theorem 1.1, p.188)\label{chowzero}
Let $X$ be a smooth projective variety over a perfect field $k$ of characteristic $p>0$.
Let $K(X)$ be the function field of $X$.
If the Chow group of 0-cycels $A_0(X\times_k \overline{K(X)})$ is equal to $\mathbb{Z}$,
then the slope $[0,1)$ part of $\mathrm{H}^i_{\text{rig}}(X/K)$ is vanishing for $i>0$.
\end{theorem}

As a direct consequence, we have the following theorem.

\begin{theorem}\label{chaingivessupersingular}
Let $X$ be a rationally chain connected K3 surface over an algebraically closed field of
characteristic $p\ge 3$.  Then, $X$ is supersingular.
\end{theorem}

\begin{proof}
Since $X$ is rationally chain connected, $X\times_k\overline{K(X)}$ is rationally 
chain connected.  Thus, the result follows from Theorem~\ref{chowzero} and Theorem~\ref{weakrationallychainconnectedK3}
since the rigid cohomology and crystalline cohomology agree on $X$.
\end{proof}

We note that, when $p=2$, a K3 surface is supersingular if and only if it is unirational (combine the results in \cite{sharv},
\cite{sharv2}, \cite{artin} and \cite{milne}).  It is known that unirational surfaces are supersingular (\cite{shioda2}, Corollary 2, p.235).  However, 
Theorem~\ref{chaingivessupersingular} shows that a much weaker condition is sufficient to obtain
the supersingularity for K3 surfaces, i.e., rational-chain-connectivity.  
This leads us to consider the equivalence of supersingularity and rational-chain-connectivity for higher
dimensional varieties, for example, threefolds.  As we will see in Section~\ref{examplesunirational},
this is true for the Calabi-Yau threefolds constructed in \cite{stefan} since they are unirational.

As a corollary of Theorem~\ref{chaingivessupersingular}, 
we have the following result.

\begin{corollary}\label{equivalencefork3surfaces}
Let $X$ be a K3 surface over an algebraically closed field $k$ of characteristic $p\ge 3$.
Assume that Conjecture~\ref{artinconjecture} is ture.  
Then, the following are equivalent:
\begin{itemize}
\item{$X$ is supersingular;}
\item{$X$ is unirational;}
\item{$X$ is rationally connected;}
\item{$X$ is rationally chain connected.}
\end{itemize}
In particular, these conditions are all equivalent if $X$ is a Kummer surface.
\end{corollary}

\begin{proof}
Suppose that $X$ is rationally chain connected.  Then, by Theorem~\ref{chaingivessupersingular}, $X$ is supersingular.  
Thus, $X$ is unirational if Conjecture~\ref{artinconjecture} holds for $X$, so $X$ is rationally connected.  
Hence $X$ is rationally chain connected.  Moreover, Conjecture~\ref{artinconjecture} holds when $X$ is a Kummer surface and $\chara k\ge 3$ (\cite{shioda}, Theorem 1.1, p.154), 
so the conditions are all equivalent for Kummer surfaces.
\end{proof}

\section{Schr\"oer's examples}\label{examplesunirational}

In this section, we review the constructions of Schr\"oer's examples and show that they
are unirational.  Note that by construction these examples exist only in characteristics two and three.

Fix an algebraically closed field $k$ of characteristic $p>0$.
Let $A$ be a superspecial Abelian surface.  Denote $X'$ by the variety $A\times_k\mathbb{P}_k^1$.  
As explained in \cite{stefan}, fix an integer $n\ge 1$ and an exact sequence 
\[
0\rightarrow \mathcal{O}_{\mathbb{P}^1_k}(-n)\overset{r,s}{\xrightarrow{\hspace*{0.2cm}}} 
\mathcal{O}_{\mathbb{P}_k^1}^{\oplus 2}\rightarrow \mathcal{O}_{\mathbb{P}_k^1}(n)\rightarrow 0 \label{exactsequence} \tag{$\dag$}
\]
given by two homogeneous quadratic polynomials $r$, $s\in H^0(\mathbb{P}_k^1, \mathcal{O}_{\mathbb{P}_k^1}(n))$
without common zeros.  Let $H\subset X'$ be the relative radical subgroup scheme of height one whose
$p$-Lie algebra is isomorphic to $\mathcal{O}_{\mathbb{P}^1_k}(-n)\subset \mathcal{O}_{\mathbb{P}_k^1}^{\oplus 2}$ (\cite{stefan}, p.1583).
Then, the quotient $X=X'/H$ has $\omega_X=\mathcal{O}_X$ if $p=3$ and $n=1$, or $p=2$ and $n=2$ (\cite{stefan}, Corollary 3.2, p.1584).  
And the quotient $X$ is an Abelian scheme over $\mathbb{P}_k^1$, whose fibers $X_t$,
$t\in \mathbb{P}_k^1$, are supersingular Abelian surfaces.

\subsection{Characteristic three case}

Now, take $p=3$ and $n=1$ in \eqref{exactsequence} as in \cite{stefan}, Section 4, p.1584.  Construct $Z\to \mathbb{P}_k^1$
as the quotient of $X$ by the involution $[-1]$.  Denote by $f: X\to Z$ the quotient $\mathbb{P}_k^1$-morphism.
For every $t\in \mathbb{P}_k^1$, the fiber $Z_t$ is a singular Kummer surface
with 16 singular points that are rational double points of type $A_1$.  
Let $D\subset Z$ be the schematic closure of the union of singular points.
Then, the blowing up with center $D$ of $Z$, $Y\to Z$, gives a smooth Calabi-Yau threefold, 
and the geometric fibers $Y_{\overline{t}}$, $t\in\mathbb{P}_k^1$, are supersingular Kummer K3 surfaces (\cite{stefan}, Proposition 4.1 and
Proposition 4.2, p.1585).

\begin{theorem}\label{unirationalp3}
Keep the notations in \cite{stefan}, Section 4, p.1584.  The Calabi-Yau threefold $Y$
constructed above is unirational.
\end{theorem}

\begin{proof}
Since $X'\to X$ is a homomorphism of group schemes, we have the commutative diagram
\[\xymatrix{
X'\ar[r]\ar[d]_{f'}  &  X\ar[d]^f  \\
Z'\ar[r]  &  Z 
}\]
where $Z'$ is the quotient of $X'$ by $[-1]$, and $Z'\to Z$ is surjective.
Note that $Z'$ is a constant family of singular Kummer surfaces.
Blow up the singular locus of $Z'$ and $Z$ to get $Y'$ and $Y$.  
There is a domaint rational map from $Y'$ to $Y$.  The smooth threefold $Y'$ is a constant 
family of Kummer K3 surfaces.  By \cite{shioda}, Theorem 1.1, p.154, the Kummer surface 
associated to a supersingular Abelian surface is unirational when $\chara k>2$.
Thus, there is a rational dominant map from $\mathbb{P}_k^2\times_k\mathbb{P}_k^1$
to $Y'$.  Since the variety $\mathbb{P}_k^2\times_k\mathbb{P}_k^1$ is rational, $Y'$ is unirational.
Therefore, $Y$ is unirational.
\end{proof}

\begin{remark}
By \cite{stefan}, Proposition 8.1, p.1590, the Artin-Mazur formal group
of the Calabi-Yau threefold constructed in \cite{stefan}, Section 4 and Section 7,
is isomorphic to the formal additive group $\hat{\mathbb{G}}_a$, whose height is
infinity.  Therefore, Theorem~\ref{unirationalp3} gives examples of $\mathrm{H}^3$-supersingular
Calabi-Yau threefolds that are unirational when $\chara k=3$.
Moreover, by Theorem~\ref{chowzero}, 
$Y$ is supersingular as defined in Definition~\ref{supersingularitydefinition}.
\end{remark}

\subsection{Characteristic two case}

Let $k$ be an algebraically closed field of characteristic 2.  Let $E$ be the supersingular elliptic curve
given by the Weierstrass equation $x^3=y^2+y$.  Take a primitive third root of unity $\zeta\in k$.
The automorphism $\varphi: E\to E$ via $(x, y)\mapsto (\zeta x, y)$ gives an action of $G=\mathbb{Z}/3\mathbb{Z}$
on $E$.  The action has three fixed points, $(0, 0)$, $(0, 1)$ and $\infty$.  Consider the supersingular
Abelian surface $A=E\times E$, endowed with the action of $G$ via $\phi=(\varphi, \varphi)$.
The morphsim $\phi: A\to A$ has $9$ fixed points, which correspondes to the singularities on $A/G$.
Note that $A/G$ is a proper normal surface over $k$.  

Let $X'$ be $A\times_k\mathbb{P}_k^1$.  Take $p=2$ and $n=2$ in \eqref{exactsequence}.  
The quotient $X=X'/H$ has $\omega_X\simeq \mathcal{O}_X$.
The fiberwise action of $G$ on $X'\to \mathbb{P}_k^1$ descends
to a fiberwise action on $X\to \mathbb{P}_k^1$.  Set $Z=X/G$ and let $Y\to Z$ be the minimal 
resolution of singularities.  Then, $Y$ is a Calabi-Yau threefold in characteristic two, and every 
geometric fiber $Y_{\overline{t}}$ of $Y\to \mathbb{P}_k^1$ has Picard number $\rho(Y_{\overline{t}})=22$.

\begin{theorem}\label{unirational2}
Keep the notations as above (\cite{stefan}, Section 7, p.1589).  The Calabi-Yau threefold $Y$
constructed above is unirational.
\end{theorem}

\begin{proof}
Let $x$ be the origin of $A$.  
Since $A$ is smooth, $\widehat{\mathcal{O}}_{A,x}$ is isomorphic
to $k[[u, v]]$ as $k$-algebras, where $u$, $v$ is a regular system of parameters of $\widehat{\mathcal{O}}_{A,x}$.
As remarked in Proposition 5.1, \cite{stefan}, p.1585, the action of $\phi$
on $\widehat{\mathcal{O}}_{A, x}$ is given by the multiplication by $\zeta$.  
So $\phi$ maps $u$ to $\zeta u$ and $v$ to $\zeta v$.
Thus, $\mathcal{O}_{A,x}/G$ is not Gorenstein.  Then, the corresponding singular point of $x$
on $A/G$ is not a rational double point since rational double point is Gorenstein. 
By \cite{kat}, Theorem 2.11, p.10, $A/G$ is a rational surface.
Then, the same argument as in Theorem~\ref{unirationalp3} shows that
the Calabi-Yau threefold $Y$ is unirational.
\end{proof}

\section{Rationally chain connected fibrations}\label{rationallyconnectedfibrations}

In this section, we prove that if the generic fiber of a fibration is coverd by
rational curves, then all other fibers are also covered by rational curves.

\begin{proposition}\label{rationallyconnectedgeometricfiber}
Let $\pi: X\to S$ be a projective morphism of irreducible Noetherian schemes.
Denote by $K$ the function field of $S$.  Assume that $X_{\overline{K}}$
is irreducible and rationally chain connected.  Then, for every $s\in S$, 
the geometric fiber $X_{\overline{\kappa(s)}}$ is also rationally chain connected.
\end{proposition}

\begin{proof}
Fix a very ample invertible sheaf $\mathcal{O}_X(1)$.  Let 
\[\hilb_{rat}(X/S)=\bigcup_e \hilb_{et+1}(X/S)\]
be the Hilbert scheme of connected tree of rational curves on $X$, where $\hilb_{et+1}(X/S)$
parametrizes trees of rational curves of degree $e$, and $\hilb_{et+1}(X/S)$ is projective 
over $S$ for every $e$.  Denote by $\mathcal{C}_e$ the universal family of trees of rational 
curves over $\hilb_{et+1}(X/S)$.  That is, we have the following diagram.

\[\xymatrix{
\mathcal{C}_e \ar[r]\ar[dr]  &  \hilb_{et+1}(X/S)\times_S X\ar[d]\ar[r]  &  X\ar[d]^{\pi}  \\
  &  \hilb_{et+1}(X/S)\ar[r]  &  S
}\]
Base change the above diagram to $\overline{K}$,
we get
\[\xymatrix{
\mathcal{C}_e\times_S\overline{K} \ar[r]\ar[dr]  &  \hilb_{et+1}(X_{\overline{K}}/\overline{K})\times_{\overline{K}} X_{\overline{K}}\ar[d]\ar[r]  
&  X_{\overline{K}}\ar[d]^{\pi}  \\
  &  \hilb_{et+1}(X_{\overline{K}}/ \overline{K})\ar[r]  &  \spec \overline{K}
}\]
where $\mathcal{C}_e\times_S\overline{K}$ is the universal family of trees of
rational curves over $\hilb_{et+1}(X_{\overline{K}}/ \overline{K})$.
Since $X_{\overline{K}}$ is rationally chain connected, a general pair of points, 
not necessarily closed,  in $X_{\overline{K}}$
can be joined by a tree of rational curves.  Thus, the image 
of $\mathcal{C}_e\times_S\overline{K}$ in $X_{\overline{K}}$ contains an open
dense subset.  Since $\mathcal{C}_e\times_S\overline{K}$ is proper over $S$, the morphism
from $\mathcal{C}_e\times_S\overline{K}$ to $X_{\overline{K}}$ is surjective.
Therefore, the morphism $\mathcal{C}_e\times_S K\to X_K$ is surjective.
So the image of the proper scheme $\mathcal{C}_e/S$ in $X$ contains the generic point
of $X$.  Then, the universal family of trees of rational curves $\mathcal{C}_e$ maps 
surjectively onto the scheme $X$.  This gives that every pair of points 
in $X_{\overline{\kappa(s)}}$ can be connected by a tree of rational curves.
\end{proof}

\begin{corollary}\label{rationallyconnectedgenericfiber}
Fix a field $k$.
Let $\pi: X\to \mathbb{P}_k^1$ be a projective morphism from an irreducible variety $X$.
Denote by $K$ the function field of $\mathbb{P}_k^1$.  Assume that $X_{\overline{K}}$
is irreducible, rationally chain connected, and there is a section of $\pi$.
Then, $X$ is rationally chain connected.
\end{corollary}

\begin{proof}
This directly follows from Definition~\ref{rationallychainconnecteddefinition}
and Proposition~\ref{rationallyconnectedgeometricfiber}.
\end{proof}

\section{K3 fibrations}\label{technicalsection}

Let \(k\) be a perfect field of characteristic \(p>0\).
Let \(K\) be the field of fractions of \(W(k)\).
Let \(\pi: X \to \mathbb{P}^{1}_{k}\) be a \emph{smooth},  \emph{projective} morphism over
\(k\).

Then we can define convergent F-isocrystals \(\mathcal{E}^i\) on
\(\mathbb{P}^{1}_k\) subject to the following conditions:
\begin{enumerate}[label=\alph*)]
\item for any perfect-field-valued point \(t: \mathrm{Spec}(F)\to X\) of
\(\mathbb{P}^1_k\), \(t^{\ast}\mathcal{E}^{i}=\mathrm{H}_{\text{rig}}^{i}(X_t)\)
(rigid cohomology over the field \(W(F)[1/p]\)),
\item there exists a spectral sequence with
\begin{equation}
\label{eq:1}\tag{\(\ast\)}
E_2^{i,j}=\mathrm{H}^{j}_{\text{rig}}(\mathbb{P}^{1}_k,\mathcal{E}^{j})
\Rightarrow \mathrm{H}^{i+j}_{\text{rig}}(X)
\end{equation}
\end{enumerate}
(See Remark 2.8 of \cite{Morrow} for details). Note that our hypotheses, smooth and projective, 
are much stronger than semi-stable as needed in \cite{Morrow}.

\begin{lemma}\label{org212707e}
In the situation above, the convergent F-isocrystals \(\mathcal{E}^{i}\)
are trivial.
\end{lemma}

\begin{proof}
Since \(\mathbb{P}^{1}_{k}\) admits a lift \(\mathbb{P}^{1}_K\), a convergent
isocrystal \(\mathcal{E}\) on \(\mathbb{P}^{1}_{k}\) is the same as a
(necessarily integrable) connection \((E^{\text{an}},\nabla)\) on the rigid analytic space
\(\mathbb{P}^{1,\text{an}}_{K}\). With some translation between formal and rigid
geometry, this assertion is precisely Proposition 1.15 of \cite{Ogus}.

Since \(\mathbb{P}^{1,\text{an}}_{K}\) is smooth and proper, the rigid analytic
GAGA theorem (cf. \cite{Kopf}) implies that \(E^{\text{an}}\) comes from an algebraic vector bundle \(E\). If
we can show that \(E\) already has an
algebraic connection, then choosing a connection on \(E\) amounts to choosing an
\(\textit{End}(E^{\text{an}})\)-valued \(1\)-form on
\(\mathbb{P}^{1,\text{an}}\); the connection is thereby necessarily algebraic by
applying GAGA again.

Hence, one needs to show that a vector bundle \(E\) on \(\mathbb{P}^{1}_{K}\)
admits an algebraic connection, given that \(E^{\text{an}}\) admits a
connection. This follows from the fact that in both rigid and algebraic
geometry, only trivial vector bundles on \(\mathbb{P}^{1}\) allow connections.

Finally, the algebraic connection \((E,\nabla)\) can be defined on some finitely
generated field and is seen to be trivial by embedding the field to
\(\mathbb{C}\) and using a complex analytic argument.
\end{proof}

\begin{proposition}\label{supersingularfibration}
Let $X$ be a smooth and projective threefold over a perfect field $k$ in 
characteristic $p>0$.  Assume that $X$ admits a K3-fibration and 
the slopes of
\(\mathrm{H}^{2}_{\text{rig}}(X)\) are all \(1\), then all the geometric fibers
of \(\pi\) are supersingular K3 surfaces.
\end{proposition}

\begin{proof}
By the K3 hypothesis, \(\mathcal{E}^1=\mathcal{E}^3=0\).
Thus in the sepectral sequence~\eqref{eq:1},
\(E_{2}^{2,1} = E_{2}^{0,3}=E_{2}^{0,1}=E_{2}^{1,1}=0\).
By Lemma~\ref{org212707e}, \(\mathcal{E}^2\) is a constant crystal, so
\begin{equation*}
E_2^{1,2}=\mathrm{H}^{1}_{\text{rig}}(\mathbb{P}^1_k,\mathcal{E}^2) = 0.
\end{equation*}
By Property a) and the constancy of \(\mathcal{E}^2\), all the \(k\)-valued
fibers \(S\) of \(\pi\) have the same rigid cohomology. Thus
\begin{equation*}
E_2^{0,2} = \mathrm{H}^0_{\text{rig}}(\mathbb{P}^1_k,\mathcal{E}^2) = \mathrm{H}_{\text{rig}}^{2}(S).
\end{equation*}
It follows from the above vanishing that the spectral sequence~\eqref{eq:1}
degenerates at \(E_2\). Thus
\(\mathrm{H}^2_{\text{rig}}(X)=\mathrm{H}^{2}_{\text{rig}}(S)\) for any fiber
\(S\) of \(\pi\). Since all the slopes of \(\mathrm{H}^2_{\text{rig}}(X)\) are
\(1\), it follows that any \(k\)-valued fiber \(S\) of \(\pi\) is a
supersingular K3 surface. The assertion for any geometric fiber then follows
from Grothendieck's specialization theorem, spelled out in our needed form in by
Theorem 2.1 of \cite{Crew}.
\end{proof}


\section{Proof of the main theorem}\label{proofofthemaintheorem}

\begin{proof}
If $X$ is rationally chain connected, then the slopes of 
\(\mathrm{H}^{2}_{\text{cris}}(X)\) are all \(1\) by Theorem~\ref{weakrationallychainconnectedK3} and Theorem~\ref{chowzero}.

Conversely, suppose that the slopes of \(\mathrm{H}^{2}_{\text{cris}}(X)\) are all \(1\).  
Let $\pi: X\to \mathbb{P}_k^1$ be the K3-fibration.
By Proposition~\ref{supersingularfibration}, all the geometric fibers
of \(\pi\) are supersingular K3 surfaces.  In particular, $X_{\overline{k(t)}}$ is a supersingular 
Kummer surface.  By \cite{shioda} Theorem 1.1 on p.154, $X_{\overline{k(t)}}$ is unirational, so rationally connected.
Since the Kummer surface $X_{k(t)}$ is the minimal resolution of a quotient of an Abelian surface over $k(t)$ under the involution, $X_{k(t)}$
has a $k(t)$-point, which gives a section of $\pi$.  Therefore, by Corollary~\ref{rationallyconnectedgenericfiber},
the total space $X$ is rationally chain connected, so uniruled (\cite{Kollar}, (3.3.4), p.200).
\end{proof}

\begin{remark}
Theorem~\ref{maintheorem} also applies to the following situation.  Let $k$ be an algebraically closed field $\chara k\ge 3$.
Let $X$ be a smooth, projective threefold over $k$ with a K3-fibration.  
Denote by $k(t)$ the function field of $\mathbb{P}_k^1$.
Assume that $X\to \mathbb{P}_k^1$ has a section and Conjecture~\ref{artinconjecture} holds.  
Then, $X$ is rationally chain connected if and only if 
the slopes of \(\mathrm{H}^{2}_{\text{cris}}(X)\) are all \(1\)
if and only if it is supersingular.
\end{remark}


 \vspace{1em}

\noindent\small{\textsc{Mathematics Department, Stony Brook University, Stony Brook, NY, 11794} }

\noindent\small{Email: \texttt{santai.qu@stonybrook.edu}}

 \end{document}